\documentclass[final,leqno]{siamltex704}
\usepackage{amsmath}
\usepackage{graphicx}
\usepackage[notcite,notref]{showkeys}
\usepackage{mathrsfs}
\usepackage{float}
\usepackage{amsfonts,amssymb}
\usepackage{dsfont}
\usepackage{pifont}
\usepackage{hyperref}
\usepackage{multirow}
\usepackage{color}
\newtheorem{defi}{Definition}[section]

\newtheorem{algorithm}[theorem]{WG-MFEM Scheme}
\newtheorem{algorithm-HMWG}[theorem]{Hybridized WG-MFEM Scheme}

\newcommand{\bq}{{\bf q}}

\newcommand{\bx}{{\bf x}}

\newcommand{\be}{{\bf e}}
\newcommand{\br}{{\bf r}}

\newcommand{\bv}{{\bf v}}

\def\T{{\mathcal T}}
\def\E{{\mathcal E}}

\def\M{{\mathcal M}}

\def\mH{{\mathbb{H}}}
\def\jump#1{{[\![#1[\!]}}

\def\bn{{\bf n}}
\def\bq{{\bf q}}

\def\l{{\langle}}
\def\r{{\rangle}}

\newcommand{\pT}{{\partial T}}

\def\3bar{{|\hspace{-.02in}|\hspace{-.02in}|}}
\def\bbQ{{\mathbb{Q}}}

\title{A Hybridized Formulation for the Weak Galerkin Mixed Finite Element Method}
\author{Lin Mu\thanks{Department of Mathematics, Michigan State University,
      East Lansing, MI 48824 (linmu@msu.edu)} \and Junping Wang\thanks{Division of Mathematical Sciences, National
Science Foundation, Arlington, VA 22230 (jwang@\break nsf.gov). The
research of Wang was supported by the NSF IR/D program, while
working at the Foundation. However, any opinion, finding, and
conclusions or recommendations expressed in this material are those
of the author and do not necessarily reflect the views of the
National Science Foundation.} \and Xiu Ye\thanks{Department of
Mathematics, University of Arkansas at Little Rock, Little Rock, AR
72204 (xxye@ualr.edu). This research was supported in part by
National Science Foundation Grant DMS-1115097.}}
\begin{document}
\maketitle

\begin{abstract}
This paper presents a hybridized formulation for the weak Galerkin
mixed finite element method (WG-MFEM) which was introduced and
analyzed in \cite{wy-m} for second order elliptic equations. The
WG-MFEM method was designed by using discontinuous piecewise
polynomials on finite element partitions consisting of polygonal or
polyhedral elements of arbitrary shape. The key to WG-MFEM is the
use of a discrete weak divergence operator which is defined and
computed by solving inexpensive problems locally on each element.
The hybridized formulation of this paper leads to a significantly
reduced system of linear equations involving only the unknowns
arising from the Lagrange multiplier in hybridization. Optimal-order
error estimates are derived for the hybridized WG-MFEM
approximations. Some numerical results are reported to confirm the
theory and a superconvergence for the Lagrange multiplier.
\end{abstract}


\pagestyle{myheadings}

\section{Introduction}\label{Section:Introduction}

In this paper, we are concerned with new developments of weak
Galerkin finite element methods for partial differential equations.
Weak Galerkin (WG) \cite{wy, wy-m, wy1302, ww-survey} is a generic
finite element method for partial differential equations where the
differential operators (e.g., gradient, divergence, curl, Laplacian
etc) in the variational form are approximated by weak forms as
generalized distributions. This process often involves the solution
of inexpensive problems defined locally on each element. The
solution from the local problems can be regarded as a reconstruction
of the corresponding differential operators. The fundamental
difference between the weak Galerkin finite element method and other
existing methods is the use of weak functions and weak derivatives
(i.e., locally reconstructed differential operators) in the design
of numerical schemes based on existing variational forms for the
underlying PDE problems. Weak Galerkin is, therefore, a natural
extension of the conforming Galerkin finite element method. Due to
its great structural flexibility, the weak Galerkin finite element
method fits well to most partial differential equations by providing
the needed stability and accuracy in approximation.

The goal of this paper is to develop a new computational method
which reduces the computational complexity for the weak Galerkin
mixed finite element methods \cite{wy-m} by using the well-known
hybridization technique \cite{fv1965,ab}. Our model problem seeks a
vector-valued function $\bq=\bq(\bx)$, also known as flux function,
and a scalar function $u=u(\bx)$ defined in an open bounded
polygonal or polyhedral domain $\Omega\subset\mathbb{R}^d\; (d=2,3)$
satisfying
\begin{eqnarray}
\alpha\bq+\nabla u=0,\ \nabla\cdot \bq=f,\quad
\mbox{in}\;\Omega\label{mix}
\end{eqnarray}
and the following Dirichlet boundary condition
\begin{equation}
 u=g, \quad \mbox{on}\; \partial\Omega,\label{bc1}
\end{equation}
where $\alpha=(\alpha_{ij}(\bx))_{d\times d}\in
[L^{\infty}(\Omega)]^{d\times d}$ is a symmetric, uniformly positive
definite matrix on the domain $\Omega$. A weak formulation for
(\ref{mix})-(\ref{bc1}) seeks $\bq\in H(div,\Omega)$ and $u\in
L^2(\Omega)$ such that
\begin{eqnarray}
(\alpha\bq,\bv)-(\nabla\cdot\bv,u)&=&-\langle
g,\bv\cdot\bn\rangle_{\partial\Omega},
\quad\forall\bv\in H(div,\Omega)\label{w-mix1}\\
(\nabla\cdot\bq,w)&=&(f,w),\quad\forall w\in L^2(\Omega).\label{w-mix2}
\end{eqnarray}
Here $L^2(\Omega)$ is the standard space of square integrable
functions on $\Omega$, $\nabla\cdot\bv$ is the divergence of
vector-valued functions $\bv$ on $\Omega$, $H(div,\Omega)$ is the
Sobolev space consisting of vector-valued functions $\bv$ such that
$\bv\in [L^2(\Omega)]^d$ and $\nabla\cdot\bv \in L^2(\Omega)$,
$(\cdot,\cdot)$ stands for the $L^2$-inner product in $L^2(\Omega)$,
and $\langle\cdot,\cdot\rangle_{\partial\Omega}$ is the inner
product in $L^2(\partial\Omega)$.

Conforming Galerkin finite element methods based on the weak
formulation (\ref{w-mix1})-(\ref{w-mix2}) and finite dimensional
subspaces of $H(div,\Omega)\times L^2(\Omega)$ with piecewise
polynomials are known as mixed finite element methods (MFEM)
\cite{rt,fv1965}. MFEMs for (\ref{mix})-(\ref{bc1}) treat $\bq$ and
$u$ as independent unknown functions and are capable of providing
accurate approximations for both unknowns
\cite{ab,babuska,brezzi,sue,bf,bddf,bdm,rt,wang}. All the existing
MFEMs in literature possess local mass conservation that makes MFEM
a competitive numerical technique in many applications such as flow
of fluid in porous media including oil reservoir and groundwater
contamination simulation.

Based on the weak formulation (\ref{w-mix1})-(\ref{w-mix2}), a weak
Galerkin mixed finite element method (WG-MFEM) was developed in
\cite{wy-m} which provides accurate numerical approximations for
both the scalar and the flux variables. Like the existing MFEMs, the
WG-MFEM scheme conserves mass locally on each element. But unlike
the existing MFEMs, the WG-MFEM allows the use of finite element
partitions consisting of polygons ($d=2$) or polyhedra ($d=3$) of
arbitrary shape that satisfy the shape-regularity assumption
specified in \cite{wy-m}. It should be pointed out that some similar
features are shared by a number of recently developed numerical
methods such as Virtual Element Methods \cite{bbcmmr,bbm},
hybridizable discontinuous Galerkin methods \cite{cgl}, and mimetic
finite differences \cite{blm1,blm2}.

While weak functions and weak derivatives provide a great deal of
flexibility for WG methods, they also introduce more degrees of
freedom than the standard finite element method. The purpose of this
paper is to develop a hybridized formulation for the weak Galerkin
mixed finite element method \cite{wy-m} that shall reduce the
computational complexity significantly by solving a linear system
involving a small number of unknowns arising from an auxiliary
function called Lagrange multiplier. The Lagrange multiplier is
defined only on the element boundaries (also called wired-basket of
the finite element partition). As a result, the linear system that
costs the majority of the computing time depends on the dimension of
the finite element space defined on the wired-basket. Optimal-order
error estimates for the hybridized WG-MFEM (HWG-MFEM) approximations
are established in several discrete norms. Some numerical results
are presented to demonstrate the efficiency and power of the
hybridized WG-FEM.

Throughout the paper, we will follow the usual notation for Sobolev
spaces and norms \cite{ciarlet,bf,sue}. For any open bounded domain
$D\subset \mathbb{R}^d$ with Lipschitz continuous boundary, we use
$\|\cdot\|_{s,D}$ and $|\cdot|_{s,D}$ to denote the norm and
seminorms in the Sobolev space $H^s(D)$ for any $s\ge 0$,
respectively. The inner product in $H^s(D)$ is denoted by
$(\cdot,\cdot)_{s,D}$. The space $H^0(D)$ coincides with $L^2(D)$,
for which the norm and the inner product are denoted by $\|\cdot
\|_{D}$ and $(\cdot,\cdot)_{D}$, respectively. When $D=\Omega$, we
shall drop the subscript $D$ in the norm and inner product notation.

The paper is organized as follows. In Section 2, we review the
discrete weak divergence operator. In Section 3, we present a
hybridized weak Galerkin mixed finite element method. Section 4 is
devoted to a discussion of the relation between the WG-MFEM and its
hybridized version. In Section 5, we derive an error estimate for
the hybridized WG-MFEM. Finally, in Section 6, we report some
numerical results that demonstrate the efficiency and accuracy of
the hybridized WG-MFEM, including a superconvergence for the
Lagrange multiplier.

\section{Weak Divergence}\label{section2}

Let $K$ be a polygonal or polyhedral domain. A weak vector-valued
function on $K$ refers to a vector field $\bv=\{\bv_0, \bv_b\}$ with
$\bv_0\in [L^2(K)]^d$ and $\bv_b\in [L^2(\partial K)]^d$. The first
component $\bv_0$ carries the information of $\bv$ in $K$, and
$\bv_b$ represents partial or full information of $\bv$ on $\partial
K$. The choice of the boundary information that $\bv_b$ represents
is problem-dependent. Note that $\bv_b$ may not necessarily be
related to the trace of $\bv_0$ on $\partial K$ should a trace be
well-defined.

Denote by $V(K)$ the space of all weak vector-valued functions on
$K$; i.e.,
\begin{equation}\label{hi.888}
V(K) = \{\bv=\{\bv_0, \bv_b \}:\ \bv_0\in [L^2(K)]^d,\; \bv_b\in
[L^2(K)]^d\}.
\end{equation}
A weak divergence can be taken for any vector field in $V(K)$ by
following the definition introduced in \cite{wy-m}, which we
summarize as follows.

\begin{defi} \cite{wy-m}
For any $\bv\in V(K)$, the weak divergence of $\bv$ is defined as a
linear functional, denoted by $\nabla_w \cdot\bv$, on $H^1(K)$ whose
action on each $\varphi\in H^1(K)$ is given by
\begin{equation}\label{weak-divergence}
\langle\nabla_w\cdot\bv, \varphi\rangle_K := -(\bv_0,
\nabla\varphi)_K + \langle \bv_b\cdot\bn, \varphi\rangle_{\partial
K},
\end{equation}
where $(\cdot,\cdot)_K$ and $\langle \cdot, \cdot\rangle_{\partial
K}$ stands for the $L^2$-inner product in $L^2(K)$ and $L^2(\partial
K)$, respectively.
\end{defi}
\medskip

The Sobolev space $[H^1(K)]^d$ can be embedded into the space $V(K)$
by an inclusion map $i_V: \ [H^1(K)]^d\to V(K)$ defined as follows
$$
i_V(\bq) = \{\bq|_{K}, \bq|_{\partial K}\}.
$$
With the help of the inclusion map $i_V$, the Sobolev space
$[H^1(K)]^d$ can be viewed as a subspace of $V(K)$ by identifying
each $\bq\in [H^1(K)]^d$ with $i_V(\bq)$. Analogously, a weak function $\bv=\{\bv_0,\bv_b\}\in V(K)$ is said to be
in $[H^1(K)]^d$ if it can be identified with a function $\bq\in
[H^1(K)]^d$ through the above inclusion map. It is not hard to see
that $\nabla_w\cdot\bv=\nabla\cdot\bv$ if $\bv$ is a smooth function
in $[H^1(K)]^d$.

Next, we introduce a discrete weak divergence operator by
approximating $\left(\nabla_w\ \cdot\ \right)$ in a polynomial
subspace of the dual of $H^1(K)$. To this end, for any non-negative
integer $r\ge 0$, denote by $P_{r}(K)$ the set of polynomials on $K$
with degree $r$ or less.
\begin{defi}
A discrete weak divergence $(\nabla_{w,r}\cdot)$ is defined as the
unique polynomial $(\nabla_{w,r}\cdot\bv) \in P_r(K)$ satisfying the
following equation
\begin{equation}\label{dwd}
(\nabla_{w,r}\cdot\bv, \phi)_K = -(\bv_0,\nabla\phi)_K+
\langle \bv_b\cdot\bn,  \phi\rangle_{\partial K},\qquad \forall
\phi\in P_r(K).
\end{equation}
\end{defi}

\section{Hybridized WG-MFEM}\label{wg-mfem}

Let ${\cal T}_h$ be a finite element partition of the domain
$\Omega$ consisting of polygons in two dimensions or polyhedra in
three dimensions satisfying the shape-regularity condition specified
in \cite{wy-m}. For $T\in \T_h$, denote by $h_T$ its diameter and
$h=\max_{T\in\T_h} h_T$ the meshsize of $\T_h$. The set of all edges
or flat faces in ${\cal T}_h$ is denoted as ${\cal E}_h$, with the
subset ${\cal E}_h^0={\cal E}_h\backslash\partial\Omega$ consisting
of all the interior edges or flat faces.

Let $k\ge 0$ be any integer. On each element $T\in\T_h$, we denote
by $\bn$ the outward normal direction on the boundary $\pT$, and
define two local finite element spaces $W_{k+1}(T)=P_{k+1}(T)$ and
\begin{equation}\label{EQ:localWGFEM-vec}
V_k(T)=\{\bv=\{\bv_0,\bv_b\}: \; \bv_0\in [P_k(T)]^d, \
\bv_b|_e=v_b\bn,\;  v_b\in P_k(e), \ e\subset \pT\}.
\end{equation}
On the wired-basket $\E_h$, we introduce a finite element space
using piecewise polynomials of degree $k$:
\begin{equation}\label{EQ:WGFEM-lambda}
\Lambda_h=\{\lambda:\ \lambda|_e\in P_k(e),\; \forall e\in\E_h\}.
\end{equation}
Let $\Lambda_h^0\subset\Lambda_h$ be the subspace consisting of
functions with zero boundary value
\begin{equation}\label{EQ:WGFEM-lambda0}
\Lambda_h^0=\{\lambda\in\Lambda_h;\ \lambda|_{e}=0,\ \forall
e\subset\partial\Omega\}.
\end{equation}
We further introduce an element-wise stabilizer as follows
\begin{equation}\label{s-T}
s_T(\br,\bv)  =  h_T\langle
(\br_0-\br_b)\cdot\bn,\;(\bv_0-\bv_b)\cdot\bn \rangle_{\partial
T},\qquad \br, \bv\in V_k(T).
\end{equation}

Denote by $\nabla_{w,k+1}\cdot$ the discrete weak divergence
operator in the space $V_k(T)$ computed by using (\ref{dwd}) on each
element $T$. For simplicity of notation, from now on we shall drop
the subscript $k+1$ from the notation $\nabla_{w,k+1}\cdot$ for the
discrete weak divergence.

\medskip

\begin{algorithm-HMWG}\label{algorithm2}
For an approximate solution of (\ref{mix})-(\ref{bc1}), find
$\bq_h=\{\bq_0, \bq_b\}\in V_k(T)$, $u_h\in W_{k+1}(T)$,
$\lambda_h\in\Lambda_h$ such that $\lambda_h=Q_bg$ on
$\partial\Omega$ and
\begin{eqnarray}
s_{T}(\bq_h,\bv)+(\alpha \bq_0, \bv_0)_T-(\nabla_w\cdot\bv, u_h)_T
&=&-\langle \lambda_h, \bv_b\cdot\bn\rangle_\pT,\quad\forall\bv\in V_k(T), \label{wg1}\\
(\nabla_w\cdot\bq_h, w)_T&=&(f,\;w)_T,\quad\forall w\in
W_{k+1}(T),\label{wg2} \\  \sum_{T\in\T_h}
\langle\bq_b\cdot\bn,\phi\rangle_\pT&=&0,\quad\forall\phi\in\Lambda_h^0.\label{wg3}
\end{eqnarray}
Here $(Q_bg)|_e\in P_k(e)$ is the $L^2$ projection of the boundary
value $u=g$ on the edge $e\subset\partial\Omega$.
\end{algorithm-HMWG}

\medskip

The approximation for the primal variable $u$ is given by
$\lambda_h$ on $\E_h$ and $u_h$ on each element $T$. The flux $\bq$
is approximated by $\bq_0$ on each element, and $\bq_b$ is an
approximation of the normal component of $\bq$ on the element
boundary.

\medskip

The hybridized WG mixed finite element scheme
(\ref{wg1})-(\ref{wg3}) is an analogy of the hybridized mixed finite
element method \cite{fv1965} (see \cite{ab, bf, wang} for more
details). Note that the approximation $\bq_h$ and $u_h$ are defined
locally on each element $T\in\T_h$ so that $\bq_b$ assumes
multi-values on each interior edge. More specifically, unlike the WG
mixed finite element method introduced in \cite{wy-m}, the
hybridized WG-MFEM scheme (\ref{wg1})-(\ref{wg3}) does not assume
the ``continuity" (or single-value) of $\bq_b$ on each interior edge
$e\in\E_h^0$. The new variable $\lambda_h$, known as the Lagrange
multiplier, was added here to provide the necessary ``continuity" of
$\bq_b$ on $\E_h^0$ through the equation (\ref{wg3}).

\medskip

In the rest of this section, we shall reformulate the problem
(\ref{wg1})-(\ref{wg3}) by eliminating the unknowns $\bq_h$ and
$u_h$. The resulting linear system is symmetric and positive
definite in terms of the variable $\lambda_h$. Readers are referred
to \cite{ab} for a similar result for the standard mixed finite
element method.

Denote by $\M_h$ the collection of all the local finite element
spaces:
$$
\M_h = \bigotimes_{T\in\T_h} V_k(T)\times W_{k+1}(T).
$$
For any given $\theta\in \Lambda_h$, let $\{\bq_h^\theta,
u_h^\theta\}\in \M_h$ be the unique solution of the following
problem
\begin{eqnarray}
s_{T}(\bq_h^\theta,\bv)+(\alpha \bq_0^\theta,
\bv_0)_T-(\nabla_w\cdot\bv, u_h^\theta)_T
&=&-\langle \theta, \bv_b\cdot\bn\rangle_\pT,\quad\forall\bv\in V_k(T), \label{wg1-theta}\\
(\nabla_w\cdot\bq_h^\theta, w)_T&=&0,\qquad\quad\forall w\in
W_{k+1}(T).\label{wg2-theta}
\end{eqnarray}
By setting
$$
\mathbb{H}(\theta):= \{\bq_0^\theta, \bq_b^\theta,
u_h^\theta\}
$$
we see that $\mathbb{H}$ defines a linear operator
from $\Lambda_h$ to $\M_h$. For convenience, we decompose the
operator $\mathbb{H}$ into three components $\mathbb{H} =\{\mH_0,
\mH_b, \mH_u\}$ so that
$$
\bq_0^\theta = \mH_0(\theta),\quad \bq_b^\theta =
\mH_b(\theta),\quad u_h^\theta = \mH_u(\theta).
$$
The first two components of $\mathbb{H}$ can be grouped together to
give the following operator
$$
\mathbb{H}_{\bq}=\{\mH_0, \mH_b\}.
$$
Note that the third component $\mH_u(\theta)$ is a piecewise
polynomial of degree $k+1$ which is indeed a discrete harmonic
extension of $\theta$ in the domain $\Omega$ from the wired-basket
$\E_h$.

With the help of the linear operator $\mH$, we can reformulate the
hybridized WG-MFEM scheme (\ref{wg1})-(\ref{wg3}) into one involving
only the unknown variable on the wired-basket $\E_h$.

\begin{theorem}\label{Thm:ReducedSystem} The solution
$\lambda_h\in\Lambda_h$ arising from the hybridized WG-MFEM scheme
(\ref{wg1})-(\ref{wg3}) is a solution of the following problem: Find
$\theta\in \Lambda_h$ satisfying $\lambda_h=Q_bg$ on
$\partial\Omega$ and the following equation
\begin{equation}\label{EQ:Scheme4Lambda}
\sum_{T\in\T_h}\left\{
s_{T}(\mH_\bq(\lambda_h),\mH_\bq(\phi))+(\alpha\mH_0(\lambda_h),\mH_0(\phi))_T\right\}
=(f, \mH_u(\phi)),\qquad \forall \phi\in\Lambda_h^0.
\end{equation}
\end{theorem}

\begin{proof}
Note that the equation (\ref{wg3}) has test function $\phi\in
\Lambda_h^0$. Owing to the operator $\mH$, by letting $\theta=\phi$
in (\ref{wg1-theta}), we arrive at
$$
-\langle \phi, \bv_b\cdot\bn\rangle_\pT = s_{T}
(\mH_\bq(\phi),\bv)+(\alpha \mH_0(\phi), \bv_0)_T-(\nabla_w\cdot\bv,
\mH_u(\phi))_T
$$
for all $\bv \in V_k(T)$. In particular, by setting $\bv = \bq_h$
(i.e., the solution of (\ref{wg1})-(\ref{wg2})), we obtain
\begin{eqnarray*}
-\langle \phi, \bq_b\cdot\bn\rangle_\pT &=& s_{T}
(\mH_\bq(\phi),\bq_h)+(\alpha \mH_0(\phi),
\bq_0)_T-(\nabla_w\cdot\bq_h, \mH_u(\phi))_T\\
&=& s_{T} (\bq_h,\mH_\bq(\phi))+(\alpha\bq_0,\mH_0(\phi))_T-(f,
\mH_u(\phi))_T,
\end{eqnarray*}
where we have used (\ref{wg2}) in the second line. By summing over
all the element $T$ we have from (\ref{wg3}) that
\begin{equation}\label{Eq:August-15:001}
\sum_{T\in\T_h}\left\{s_{T}
(\bq_h,\mH_\bq(\phi))+(\alpha\bq_0,\mH_0(\phi))_T \right\} = (f,
\mH_u(\phi)),\qquad \forall \phi\in\Lambda_h^0.
\end{equation}

Next, by using (\ref{wg2-theta}) and then (\ref{wg1}) with
$\bv=\mH_\bq(\phi)$ we obtain
\begin{eqnarray*}
&&s_{T}(\bq_h,\mH_\bq(\phi))+(\alpha\bq_0,\mH_0(\phi))_T \\
&=&
s_{T}(\bq_h,\mH_\bq(\phi))+(\alpha\bq_0,\mH_0(\phi))_T - (\nabla_w\cdot \mH_\bq(\phi), u_h)_T\\
&=& -\langle \lambda_h, \mH_b(\phi)\cdot\bn\rangle_\pT\\
&=&s_{T}(\mH_\bq(\lambda_h),\mH_\bq(\phi))+(\alpha\mH_0(\lambda_h),\mH_0(\phi))_T
- (\nabla_w\cdot \mH_\bq(\phi), \mH_u(\lambda_h))_T\\
&=&s_{T}(\mH_\bq(\lambda_h),\mH_\bq(\phi))+(\alpha\mH_0(\lambda_h),\mH_0(\phi))_T.
\end{eqnarray*}
Substituting the above equation into (\ref{Eq:August-15:001}) yields
$$
\sum_{T\in\T_h}\left\{
s_{T}(\mH_\bq(\lambda_h),\mH_\bq(\phi))+(\alpha\mH_0(\lambda_h),\mH_0(\phi))_T\right\}
=(f, \mH_u(\phi)),\qquad \forall \phi\in\Lambda_h^0,
$$
which is precisely the equation (\ref{EQ:Scheme4Lambda}).
\end{proof}

\medskip

Concerning the solution existence and uniqueness for the reduced
system (\ref{EQ:Scheme4Lambda}), we have the following result.

\begin{theorem}\label{Thm:ReducedSystem-Existence} There exists one and only one
$\lambda_h\in\Lambda_h$ satisfying the equation
(\ref{EQ:Scheme4Lambda}) and the boundary condition $\lambda_h=Q_bg$
on $\partial\Omega$.
\end{theorem}

\begin{proof}
Since the number of unknowns equals the number of equations, it is
sufficient to verify the solution uniqueness. To this end, let
$\lambda_h^{(i)}\in \Lambda_h$ be two solutions of
(\ref{EQ:Scheme4Lambda}) satisfying the boundary condition
$\lambda_h^{(i)}=Q_bg,\ i=1,2$. It is clear that their difference,
denoted by $\sigma=\lambda_h^{(1)}-\lambda_h^{(2)}$, is a function
in $\Lambda_h^0$ and satisfies
$$
\sum_{T\in\T_h}\left\{
s_{T}(\mH_\bq(\sigma),\mH_\bq(\phi))+(\alpha\mH_0(\sigma),\mH_0(\phi))_T\right\}
=0,\qquad \forall \phi\in\Lambda_h^0.
$$
By setting $\phi=\sigma$ we arrive at
$$
s_{T}(\mH_\bq(\sigma),\mH_\bq(\sigma))+(\alpha\mH_0(\sigma),\mH_0(\sigma))_T
=0,\qquad \forall \ T\in\T_h.
$$
It follows that $\mH_0(\sigma)=0$ on each element $T$ and
$\mH_b(\sigma)\cdot\bn = \mH_0(\sigma)\cdot\bn$ on $\pT$, which
implies $\mH_b(\sigma)=0$ on $\pT$. Now using (\ref{wg1-theta}) we
have
\begin{equation}\label{EQ:August-15:100}
(\nabla_w\cdot\bv, \mH_u(\sigma))_T = \langle \sigma,
\bv_b\cdot\bn\rangle_\pT,\quad\forall \bv\in V_k(T), \ T\in\T_h.
\end{equation}
The definition of the discrete weak divergence can be applied to the
left-hand side to yield
\begin{equation}\label{EQ:August-15:110}
-(\bv_0, \nabla \mH_u(\sigma))_T+\langle \bv_b\cdot\bn,
\mH_u(\sigma)\rangle_\pT = \langle \sigma,
\bv_b\cdot\bn\rangle_\pT,\quad\forall \bv\in V_k(T).
\end{equation}
Thus, by setting $\bv_b=0$ and varying $\bv_0$ we obtain
$$
\nabla \mH_u(\sigma) = 0\  \Longrightarrow \ \mH_u(\sigma)=const,
\quad \mbox{in}\ T\in\T_h,
$$
which, together with varying $\bv_b$ in (\ref{EQ:August-15:110}),
leads to $\mH_u(\sigma)=\sigma$ on $\pT$. It follows that both
$\mH_u(\sigma)$ and $\sigma$ are constants in the domain $\Omega$
and vanishes on the boundary. Hence, $\sigma=\mH_u(\sigma)=0$. This
completes the proof of the theorem.
\end{proof}

\section{Equivalence between HWG-MFEM and WG-MFEM}
The goal of this section is to show that the hybridized WG-MFEM
approximate solution arising from (\ref{wg1})-(\ref{wg3}) coincides
with the solution from the corresponding WG-MFEM scheme introduced
as in \cite{wy-m}. To this end, consider the finite element space
$\tilde{V}_h = \bigotimes_{T\in\T_h} V_k(T)$. For any interior
edge/face $e\in\E_h^0$, denote by $T_1$ and $T_2$ the two elements
that share $e$ as a common side. Recall that functions
$\bv=\{\bv_0,\bv_b\}\in \tilde{V}_h$ have two-sided values on $e$
for the component $\bv_b$: one comes from the element $T_1$ and
other from $T_2$. Define the jump of $\bv$ along $e\in {\cal E}_h^0$
as follows
\begin{equation}\label{similarity}
[\![\bv]\!]_e= \bv_b|_{_{\partial
T_1}}\cdot\bn_{1}+\bv_b|_{_{\partial T_2}}\cdot\bn_{2},
\end{equation}
where $\bn_i$ is the outward normal direction on $e$ as seen from
the element $T_i$, $i=1,2$.

Let $k\ge 0$ be any integer. Following \cite{wy-m}, we introduce two
weak finite element spaces
\begin{eqnarray*}
V_h&=&\{\bv=\{\bv_0,\bv_b\}: \;\bv\in \tilde{V}_h,\
[\![\bv]\!]_e=0,\ \forall e\in\E_h^0\}, \\ W_h&=& \bigotimes_{T\in
\T_h} W_{k+1}(T).
\end{eqnarray*}
On each element $T\in \T_h$, we introduce three bilinear forms
\begin{eqnarray}
a_{T}(\br,\;\bv) & = & (\alpha\br_0,\bv_0)_T,
\label{a-T}\\
b_T(\bv,\ w)& = &(\nabla_w\cdot\bv,\;w)_T,\label{b-T}\\
c_T(\bv,\ \sigma)& = &\l \bv_b\cdot\bn,\ \sigma\r_{\pT},
\end{eqnarray}
for $\bv=\{\bv_0,\bv_b\},\br=\{\br_0,\br_b\}\in  V_k(T)$, $w\in
W_{k+1}(T)$ and $\sigma\in\Lambda_h$. In addition, set
$$
a_{s,T}(\br,\;\bv)=a_T(\br,\;\bv)+s_T(\br,\;\bv).
$$

\smallskip

\begin{algorithm}\label{algorithm1} \cite{wy-m}
For an approximation of the solution of (\ref{mix})-(\ref{bc1}),
find $\bar{\bq}_h=\{\bar{\bq}_0, \bar{\bq}_b\}\in V_h$ and
$\bar{u}_h\in W_h$ such that
\begin{eqnarray}
\sum_{T\in\T_h} a_{s,T}(\bar{\bq}_h,\bv)-\sum_{T\in\T_h} b_T(\bv, \bar{u}_h)&=&-\l g,\ \bv_b\cdot\bn\r_{\partial\Omega},\quad\forall\bv\in  V_h \label{wgm1}\\
\sum_{T\in\T_h} b_T(\bar{\bq}_h, w)&=&(f,\;w), \quad\quad\forall w\in W_h.\label{wgm2}
\end{eqnarray}
\end{algorithm}

\smallskip
The following Theorem shows that the WG-MFEM is equivalent to the
hybridized WG-MFEM.

\begin{theorem}\label{wg-hwg} Let
$(\bq_h,u_h,\lambda_h)\in \M_h\times \Lambda_h$ be the hybridized
WG-MFEM approximation of (\ref{mix})-(\ref{bc1}) obtained from
(\ref{wg1})-(\ref{wg3}), and $(\bar\bq_h,\bar{u}_h)\in  V_h\times
W_h$ be the WG-MFEM approximation of (\ref{mix})-(\ref{bc1}) arising
from (\ref{wgm1})-(\ref{wgm2}). Then we have $\bq_h=\bar\bq_h$ and
$u_h=\bar{u}_h$. Moreover, the hybridized WG-MFEM scheme
(\ref{wg1})-(\ref{wg3}) has one and only one solution.
\end{theorem}

\begin{proof} Assume that $(\bq_h,u_h,\lambda_h)\in \M_h\times \Lambda_h$ is a solution
of (\ref{wg1})-(\ref{wg3}). From the equation (\ref{wg3}) we see
that $\jump{\bq_h}=0$ on each interior edge $e\in\E_h^0$. Hence,
$\bq_h\in V_h$. Moreover, for any $\bv\in V_h$ so that
$\jump{\bv_b}|_e=0$ with $e\in\E_h^0$, we have
\begin{equation}\label{eu2}
\sum_{T\in\T_h}\langle \lambda_h, \bv_b\cdot\bn\rangle_{\partial
T}=\l g,\bv_b\cdot\bn\r_{\partial\Omega}.
\end{equation}
Therefore, for $\bv\in V_h$ and $w\in W_h$, the equations
(\ref{wg1})-(\ref{wg2}) leads to
\begin{eqnarray*}
\sum_{T\in\T_h}a_{s,T}(\bq_h,\bv)-\sum_{T\in\T_h}b_T(\bv, u_h)&=&-\l g,\bv_b\cdot\bn\r_{\partial\Omega},\\
\sum_{T\in\T_h}b_T(\bq_h, w)&=&(f,\;w),
\end{eqnarray*}
which shows that $(\bq_h,u_h)$ is a solution for the WG-MFEM. The
solution existence and uniqueness for the WG-MFEM \cite{wy-m} then
implies $\bq_h=\bar\bq_h$ and $u_h=\bar{u}_h$.

To show that the hybridized WG-MFEM scheme (\ref{wg1})-(\ref{wg3})
has a unique solution, assume that
$(\bq_h^{(i)},u_h^{(i)},\lambda_h^{(i)})\in \M_h\times \Lambda_h,\
i=1,2,$ are two solutions. Their difference,
$$
(\bq_h,u_h,\lambda_h):=(\bq_h^{(1)}-\bq_h^{(2)},
u_h^{(1)}-u_h^{(2)},\lambda_h^{(1)}-\lambda_h^{(2)}),
$$
must satisfy (\ref{wg1})-(\ref{wg3}) with $f\equiv 0$ and $g\equiv
0$. From the first part of this theorem, we see that
$$
\bq_h=0,\quad u_h=0.
$$
Thus, from (\ref{wg1}) we have
\begin{equation*}
0 = \langle \lambda_h, \bv_b\cdot\bn\rangle_\pT,\quad\forall \bv\in
V_k(T), \ T\in\T_h,
\end{equation*}
which implies $\lambda_h=0$ on $\pT$ for any $T\in\T_h$. This
completes the proof of the theorem.
\end{proof}

\medskip

The hybridized WG-MFEM scheme (\ref{wg1})-(\ref{wg3}) consists of
two parts: a local system (\ref{wg1})-(\ref{wg2}) defined on each
element $T\in\T_h$ and a global system (\ref{wg3}). This scheme has
a reformulation of (\ref{EQ:Scheme4Lambda}) which involves only the
Lagrange multiplier $\lambda_h$. It should be pointed out that the
reduced system (\ref{EQ:Scheme4Lambda}) is symmetric and  positive
definite. Since the Lagrange multiplier $\lambda_h$ is defined only
on the element boundary as its unknowns, the size of the linear
system for the hybridized WG-MFEM is significantly smaller than the
one for the WG-MFEM scheme. Thus, the hybridized WG-MFEM is an
efficient implementation of the WG-MFEM scheme proposed and analyzed
in \cite{wy-m}.

\section{Error Estimates}\label{section-projections} Owing to the
equivalence between the hybridized WG-MFEM scheme
(\ref{wg1})-(\ref{wg3}) and the WG-MFEM scheme of \cite{wy-m}, all
the error estimates developed in \cite{wy-m} can be applied to the
approximate solution $(\bq_h,u_h)$ obtained from
(\ref{wg1})-(\ref{wg3}). What remains to study is the convergence
and error estimate for the Lagrange multiplier $\lambda_h$.

On each element $T\in \T_h$, denote by $Q_0$ the $L^2$ projection
from $[L^2(T)]^d$ to $[P_k(T)]^d$, by $Q_b$ the $L^2$ projection
from $L^2(e)$ to $P_k(e)$, and by $\bbQ_h$ the $L^2$ projection from
$L^2(T)$ onto $P_{k+1}(T)$. With the help of these operators, we
define a projection operator $Q_h: V(T) \to V_k(T)$ so that for any
$\bv=\{\bv_0,v_b\bn\}\in V(T)$
\begin{equation}\label{Qh-def}
Q_h \bv=\{Q_0\bv_0, (Q_bv_b)\bn\}.
\end{equation}

In the finite element space $W_h$, we introduce the following norm:
\begin{equation}\label{discreteH1norm}
\|w\|_{1,h}^2=\sum_{T\in {\cal T}_h}\|\nabla
w\|_T^2+h^{-1}\sum_{e\in\E_h} \|\jump{w}\|_e^2.
\end{equation}
In the space $V_h$, we use
$$
\3bar\bv\3bar^2=\sum_{T\in\T_h}a_{s,T}(\bv,\bv).
$$
For simplicity of notation, we shall use $\lesssim$ to denote {\em
``less than or equal to"} up to a constant independent of the mesh
size, variables, or other parameters appearing in the inequalities
of this section.

\smallskip

\begin{theorem}\label{THM:1stErrorEstimate} \cite{wy-m}
Let $(\bq_h,u_h,\lambda_h)\in \M_h\times\Lambda_h$ be the
approximate solution of (\ref{mix})-(\ref{bc1}) arising from the
hybridized WG-MFEM scheme (\ref{wg1})-(\ref{wg3}) of order $k\ge 0$.
Assume that the exact solution $(\bq, u)$ of
(\ref{w-mix1})-(\ref{w-mix2}) is regular such that $u\in
H^{k+2}(\Omega)$ and $\bq\in [H^{k+1}(\Omega)]^d$. Then, we have
\begin{eqnarray}
\3bar \bq_h-Q_h\bq\3bar+\|u_h-\bbQ_hu\|_{1,h} \lesssim
h^{k+1}\left(\|u\|_{k+2}+\|\bq\|_{k+1}\right).\label{qh-qq}
\end{eqnarray}
If, in addition, the problem (\ref{mix}) with homogeneous Dirichlet
boundary condition $u=0$ has the usual $H^2$-regularity, then the
following optimal order error estimate in $L^2$ holds true:
\begin{eqnarray}
\|u_h-\bbQ_hu\|\lesssim
h^{k+2}\left(\|u\|_{k+2}+\|\bq\|_{k+1}\right)\label{ul2}.
\end{eqnarray}
\end{theorem}

\medskip

For any $\bv=\{\bv_0, \bv_b \}\in V_k(T)$ and $w\in H^1(T)$, using
the definition of $\bbQ_h$ and the integration by parts we obtain
\begin{eqnarray}
(\nabla_w\cdot \bv,\; \bbQ_hw)_T &=& -(\bv_0,\; \nabla (\bbQ_hw))_T
+ \langle \bv_b\cdot\bn,\;
\bbQ_hw\rangle_{\partial T}\nonumber\\
&=& (\nabla\cdot \bv_0,\; \bbQ_hw)_T + \langle \bv_b\cdot\bn - \bv_0\cdot
\bn,\;
\bbQ_hw\rangle_{\partial T}\nonumber\\
&=& (\nabla\cdot \bv_0,\; w)_T + \langle \bv_b\cdot\bn - \bv_0\cdot \bn,\;
\bbQ_hw\rangle_{\partial T}\nonumber\\
&=& -(\bv_0,\; \nabla w)_T + \langle \bv_0\cdot\bn,\;
w\rangle_{\partial T}+\langle \bv_b\cdot\bn - \bv_0\cdot \bn,\; \bbQ_hw\rangle_{\partial T}\nonumber\\
&=& -(\bv_0,\; \nabla w)_T + \langle \bv_0\cdot\bn-\bv_b\cdot\bn,\;
w-\bbQ_hw\rangle_{\partial T} + \langle \bv_b\cdot\bn,\;
w\rangle_{\partial T}.\label{key111}
\end{eqnarray}
By introducing the following notation
$$
\ell_T(w,\bv)=\langle \bv_0\cdot\bn-\bv_b\cdot\bn,\;\bbQ_hw-
w\rangle_{\pT},
$$
we have from (\ref{key111}) that
\begin{eqnarray}
(\nabla_w\cdot \bv,\; \bbQ_hw)_T =-(\bv_0,\; \nabla w)_T+ \langle
\bv_b\cdot\bn,\; w\rangle_{\pT} +\ell_T(w,\bv). \label{key118}
\end{eqnarray}

\begin{lemma} Let $(\bq_h,u_h,\lambda_h)\in \M_h\times\Lambda_h$ be the
approximate solution of (\ref{mix})-(\ref{bc1}) arising from the
hybridized WG-MFEM scheme (\ref{wg1})-(\ref{wg3}) of order $k\ge 0$.
On each element $T\in\T_h$, the following identity holds true
\begin{equation}\label{err1}
\begin{split}
&s_T(\be_h,\bv)+(\alpha
\be_0,\bv_0)_T-(\nabla_w\cdot\bv,\epsilon_h)_T+\l\delta_h,
\bv_b\cdot\bn\r_\pT\\
=\ &s_T(Q_h\bq,\bv)+\ell_T(u,\bv),\qquad \forall\bv\in V_k(T),
\end{split}
\end{equation}
where
$$
\be_h=\{\be_0,\be_b\}=Q_h\bq-\bq_h,\quad
\epsilon_h=\bbQ_hu-u_h,\quad \delta_h=Q_bu-\lambda_h,
$$
stand for the error between the hybridized WG-MFEM approximate
solution and the $L^2$ projection of the exact solution.
\end{lemma}

\begin{proof}
For any $\bv=\{\bv_0,\bv_b\}\in V_k(T)$, we test the first equation
of (\ref{mix}) against $\bv_0$ to obtain
$$
(\alpha\bq,\bv_0)_T+(\nabla u,\bv_0)_T=0.
$$
By applying the identity (\ref{key118}) to the term $(\nabla u,
\bv_0)_T$ we arrive at
\begin{eqnarray*}
(\alpha\bq,\bv_0)_T-(\nabla_w\cdot \bv,\bbQ_hu)_T+\langle
\bv_b\cdot\bn,u\rangle_{\pT}=\ell_T(u,\bv).
\end{eqnarray*}
Adding $s_T(Q_h\bq,\bv)$ to both sides of the above equation and
then using the definition of $Q_h$, we have
\begin{eqnarray*}
s_T(Q_h\bq,\bv)+(\alpha (Q_0\bq),\bv_0)_T-(\nabla_w\cdot
\bv,\bbQ_hu)_T+\langle \bv_b\cdot\bn,u\rangle_{\pT}
=\ell_T(u,\bv)+s_T(Q_h\bq,\bv).
\end{eqnarray*}
Now subtracting (\ref{wg1}) from the above equation gives
(\ref{err1}). This completes the proof of the lemma.
\end{proof}

\medskip

Let $T\in\T_h$ be an element with $e$ as an edge/face. For any
function $\varphi\in H^1(T)$, the following trace inequality has
been derived for general polyhedral partitions satisfying the shape
regularity assumptions {\bf A1 - A4} (see \cite{wy-m} for details):
\begin{equation}\label{trace}
\|\varphi\|_{e}^2 \lesssim \left( h_T^{-1} \|\varphi\|_T^2 + h_T
\|\nabla \varphi\|_{T}^2\right).
\end{equation}

\begin{theorem} Under the assumptions of Theorem
\ref{THM:1stErrorEstimate}, we have the following error estimate:
\begin{equation}\label{err-l}
\left(\sum_{T\in\T_h}h_T\|Q_bu -
\lambda_h\|_{\pT}^2\right)^{\frac12}\lesssim
h^{k+2}\left(\|u\|_{k+2}+\|\bq\|_{k+1}\right).
\end{equation}
\end{theorem}

\begin{proof} By letting $\bv=\{0, \delta_h\bn\}$, we have $\bv_0=0$ and
$(\nabla_w\bv, \epsilon_h)_T = \l\delta_h, \epsilon_h\r_\pT$. It
follows from (\ref{err1}) that
\begin{eqnarray*}
\|\delta_h\|^2_{\pT}&=&\l\delta_h,\bv_b\cdot\bn\r_\pT\\
&=&-s_{T}(\be_h,\bv)+(\nabla_w\cdot\bv,\epsilon_h)_T+s_T(Q_h\bq,\bv) +\ell_T(u,\bv)\nonumber\\
&=&h_T\l \be_0\cdot\bn-\be_b\cdot\bn, \delta_h\r_\pT+\l \delta_h,\epsilon_h\r_\pT\nonumber\\
&& -h_T \l Q_0\bq\cdot\bn-Q_b(\bq\cdot\bn), \delta_h\r_\pT -\langle
\delta_h,\bbQ_hu-u\rangle_{\partial T}. 
\end{eqnarray*}
Thus, from the usual Cauchy-Schwarz inequality, we obtain
\begin{equation}\label{lambda}
\begin{split}
\|\delta_h\|_{\pT}\lesssim & \
h_T\|\be_0\cdot\bn-\be_b\cdot\bn\|_\pT
+ \|\epsilon_h\|_\pT \\
&  + h_T \|Q_0\bq\cdot\bn-Q_b(\bq\cdot\bn)\|_\pT+\|\bbQ_hu-u\|_\pT.
\end{split}
\end{equation}
By first applying the trace inequality (\ref{trace}) to the last
three terms on the right-hand side of (\ref{lambda}), and then
summing over all the element $T\in\T_h$ we obtain
\begin{equation}\label{lambda-new}
\begin{split}
\sum_{T\in\T_h} h_T\|\delta_h\|_{\pT}^2\lesssim & \
h^2\sum_{T\in\T_h} h_T\|\be_0\cdot\bn-\be_b\cdot\bn\|_\pT^2
+ \sum_{T\in\T_h} (\|\epsilon_h\|_T^2 + h_T^2\|\nabla\epsilon_h\|_T^2) \\
&  + h^2 \sum_{T\in\T_h}(\|Q_0\bq-\bq\|_T^2 +
h_T^2\|\nabla(Q_0\bq-\bq)\|_T^2)\\
& + \sum_{T\in\T_h}(\|\bbQ_h u -u\|_T^2 + h_T^2\|\nabla(\bbQ_h u
-u)\|_T^2)\\
\lesssim & \ h^2 \3bar \be_h\3bar^2 + \|\epsilon_h\|^2 +
h^2\|\epsilon_h\|_{1,h}^2 + h^{2k+4}(\|\bq\|_{k+1}^2+\|u\|_{k+2}^2).
\end{split}
\end{equation}
Combining the above inequality with the error estimates
(\ref{qh-qq}) and (\ref{ul2}) yields (\ref{err-l}). This completes
the proof of the theorem.
\end{proof}

\section{Numerical Experiments}

In this section, we present some numerical results for the
hybridized WG-MFEM scheme (\ref{wg1})-(\ref{wg3}) based on the
lowest order element (i.e., $k=0$) for the second order elliptic
problem (\ref{mix})-(\ref{bc1}). Recall that for the lowest order
hybridized WG-MFEM, the corresponding finite element spaces are
given by
\begin{eqnarray*}
\M_h &=& \bigotimes_{T\in\T_h} V_0(T)\times W_1(T),\\
\Lambda_h &=& \{\mu:\mu|_e\in P_0(e),e\in\E_h\}.
\end{eqnarray*}

For any given $\bv=\{\bv_0,\bv_b\}$, the discrete weak divergence
$\nabla_w\cdot\bv\in P_1(T)$ is computed locally on each element $T$
as follows: Find $\nabla_w\cdot \bv \in P_1(T)$ such that
$$
(\nabla_w\cdot \bv,\tau)_T=-(\bv_0,\nabla\tau)_T+\l
\bv_b\cdot\bn,\tau\r_{\partial T},\qquad \forall \tau\in P_1(T).
$$

Let $(\bq_h,u_h,\lambda_h)\in \M_h\times\Lambda_h$ be the hybridized
WG-MFEM approximate solution arising from (\ref{wg1})-(\ref{wg3})
and $(\bq,u)$ be the exact solution of (\ref{mix})-(\ref{bc1}),
respectively. In our numerical experiments, we compare the numerical
solutions with the $L^2$ projection of the exact solution in various
norms for their difference:
$$
\be_h=\{\be_0,\be_b\}=Q_h\bq-\bq_h,\ \epsilon_h=\mathbb{Q}_hu-u_h,\
\delta_h=Q_b u - \lambda_h.
$$
In particular, the following norms are used to measure the scale of
the error:
\begin{align*}
&H^1\mbox{-norm:}\quad \|\epsilon_h\|_{1,h}=\left(\sum_{T\in\mathcal{T}_h}\|\nabla
\epsilon_h\|_T^2+h^{-1}\sum_{e\in\mathcal{E}_h}\|\jump{\epsilon_h}\|_e^2\right)^{\frac12}, \\
&L^2\mbox{-norm:}\quad \3bar \be_h
\3bar=\bigg(\sum_{T\in\mathcal{T}_h}\int_T|\be_0|^2dx
+h_T\int_{\partial T}|\be_0\cdot\bn-\be_b\cdot\bn|^2ds\bigg)^{\frac12},\\
&{L^2}\mbox{-norm:}\quad \| \delta_h \|=\bigg(\sum_{T\in\mathcal{T}_h}h_T
\|\lambda_h-Q_bu\|^2_{\pT\setminus\partial\Omega}\bigg)^{\frac12},\\
&L^2\mbox{-norm:}\quad
\|\epsilon_h\|=\bigg(\sum_{T\in\mathcal{T}_h}\int_T|\epsilon_h|^2dx\bigg)^{\frac12}.
\end{align*}

\subsection{Example 1}  In this test case, the domain is the unit square
$\Omega=(0,1)\times(0,1)$ and the coefficient is
$\alpha=\frac{1}{(1+x)(1+y)}$. The exact solution is given by
$u=\sin(\pi x)\sin(\pi y)$.

\begin{table}[h!]
\caption{Example 1. Convergence rate on triangular elements.}\label{table:ex1}
\center
\begin{tabular}{||c||cc|cc|cc|cc||}
\hline\hline
$h$ & $\3bar\be_h\3bar$ & order & $\|\delta_h\|$& order  & $\|\epsilon_h\|_{1,h}$& order  & $\|\epsilon_h\|$& order \\
\hline\hline
   1/4    &2.93e-1   &    -     &{3.10e-2}   &    -     &1.22      &  -    &1.91e-1   &-   \\ \hline
   1/8    &1.47e-1   &{0.99}   &{8.17e-3}   &{1.92}   &5.07e-1   &1.27   &4.76e-2   &2.01\\ \hline
   1/16   &7.39e-2   &{1.00}   &{2.07e-3}   &{1.98}   &2.39e-1   &1.08   &1.19e-2   &2.00\\ \hline
   1/32   &3.70e-2   &{1.00}   &{5.20e-4}   &{1.99}   &1.18e-1   &1.02   &2.97e-3   &2.00\\ \hline
   1/64   &1.85e-2   &1.00      &{1.30e-4}   &{2.00}  &5.87e-2   &1.01   &7.42e-4   &2.00\\ \hline
   1/128  &9.25e-3   &1.00      &{3.25e-5}   &{2.00}      &2.93e-2   &1.00   &1.86e-4   &2.00\\ \hline\hline
\end{tabular}
\end{table}

This numerical experiment was performed on uniform triangular
partitions for the domain. The triangular partitions are constructed
as follows: 1) first uniformly partition the domain into $n\times n$
sub-rectangles; 2) then divide each rectangular element by the
diagonal line with a negative slope. The mesh size is denoted by
$h=1/n.$ Table \ref{table:ex1} shows the convergence rate for the
hybridized WG-MFEM solutions measured in different norms. The
results show that the hybridized WG-MFEM approximates are convergent
with rate $O(h)$ in $H^1$ and $O(h^2)$ in $L^2$ norms.

Recall that the primal variable $u$ is approximated by the Lagrange
multiplier $\lambda_h$ as a piecewise constant function on the
wired-basket $\E_h$. The error function $\delta_h$, which is the
difference of $\lambda_h$ and the $L^2$ projection of the exact
solution on each edge $e\in\E_h$, is shown to be convergent at the
rate of $O(h^2)$. Thus, the Lagrange multiplier is a superconvergent
approximation of the exact solution at the mid-point of each edge
$e\in\E_h$. This result is similar to the hybridized mixed finite
element method.

\subsection{Example 2} In our second test, the domain $\Omega$ is again the unit square.
But the finite element partitions are given by the uniform
rectangular meshes. The coefficient $\alpha$ is given by $\alpha=1$.
The data is chosen so that the exact solution is $u=\sin(\pi
x)\cos(\pi y)$. The numerical results are presented in Table
\ref{table:ex2}.

\begin{table}[h!]
\caption{Example 2. Convergence rate on rectangular elements.}\label{table:ex2}
\center
\begin{tabular}{||c||cc|cc|cc|cc||}
\hline\hline
$h$ & $\3bar\be_h\3bar$ & order & $\|\delta_h\|$& order  & $\|\epsilon_h\|_{1,h}$& order  & $\|\epsilon_h\|$& order \\
\hline\hline
   1/4    &9.19e-1  &   -      &{ 1.86e-2}   &   -   &2.08     &     -    & 8.70e-1   &   -\\ \hline
   1/8    &4.84e-1  &{ 0.93}   &{ 4.57e-3}   &2.03   &2.09     &     -    & 2.85e-1   &1.61\\ \hline
   1/16   &2.45e-1  &{ 0.98}   &{ 1.14e-3}   &{ 2.01}   &1.33     & { 0.65}  & 7.90e-2   &1.85\\ \hline
   1/32   &1.23e-1  &{ 1.00}   &{ 2.84e-4}   &{ 2.00}   &7.29e-1  & { 0.87}  & 2.06e-2   &1.94\\ \hline
   1/64   &6.15e-2  &{ 1.00}   &{ 7.10e-5}   &{ 2.00}   &3.78e-1  & { 0.95}  & 5.25e-3   &1.97\\ \hline
   1/128  &3.08e-2   &1.00     &{ 1.77e-5}   &{ 2.00}   &1.92e-1  & { 0.98}  & 1.31e-3   &2.00\\ \hline\hline
\end{tabular}
\end{table}

Table \ref{table:ex2} shows the optimal rate of convergence for the
numerical solution in $H^1$ and $L^2$ norms. Once again, we see a
superconvergence for the solution on the wired-basket.

\subsection{Example 3} The test problem here is the same as Example 2, but the
finite element partitions consist of quadrilaterals constructed as
follows. Starting with a coarse quadrilateral mesh shown as in
Figure \ref{fig:ex3} (Left), we successively refine each
quadrilateral by connecting its barycenter with the middle points of
its edges, shown as in Figure \ref{fig:ex3} (Right). The numerical
results are presented in Table \ref{table:ex3}. All the numerical
results are in consistency with the theory developed in this paper.

\begin{figure}[h!]
\centering
\begin{tabular}{cc}
  \resizebox{2.3in}{2in}{\includegraphics{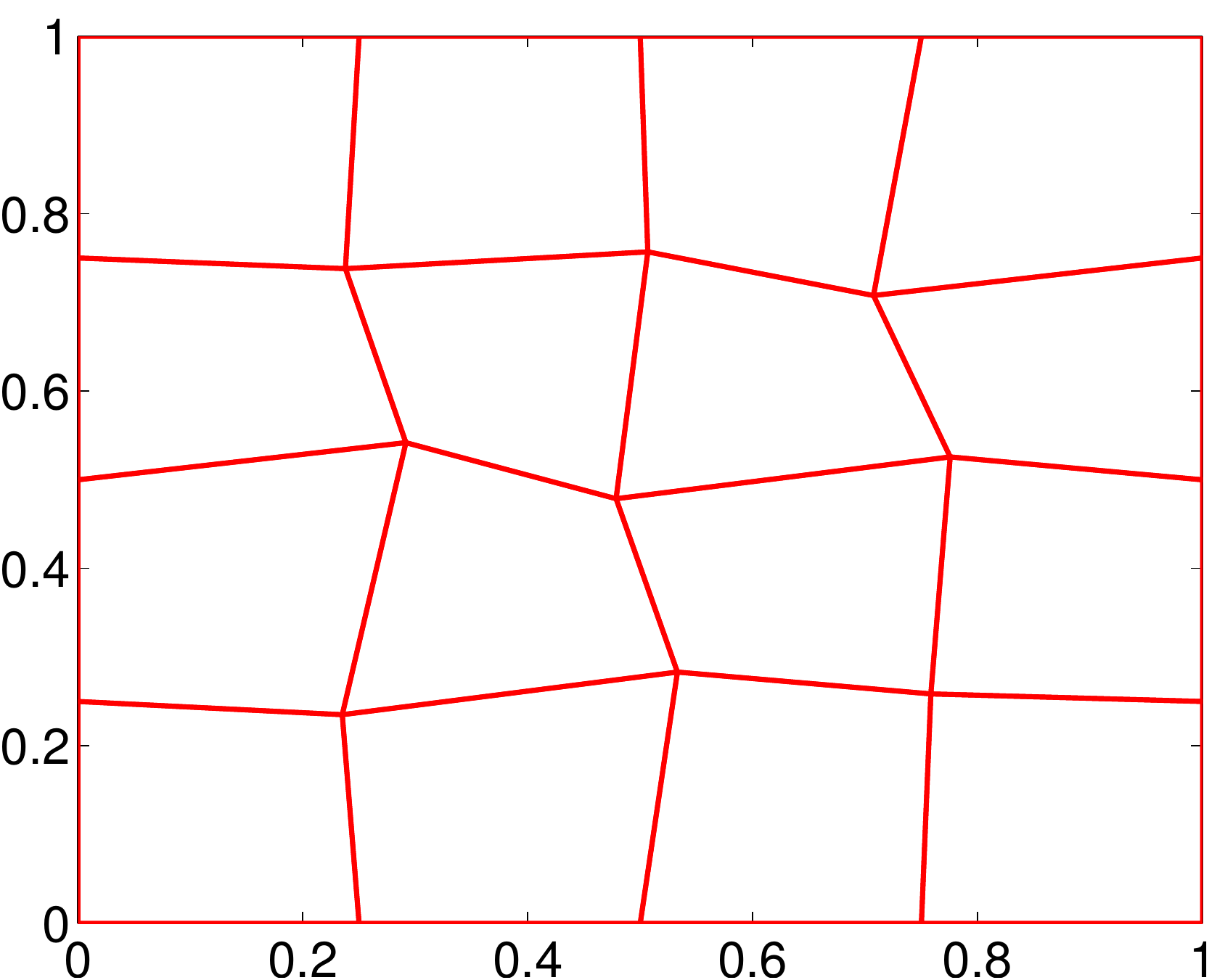}}
  \resizebox{2.4in}{2in}{\includegraphics{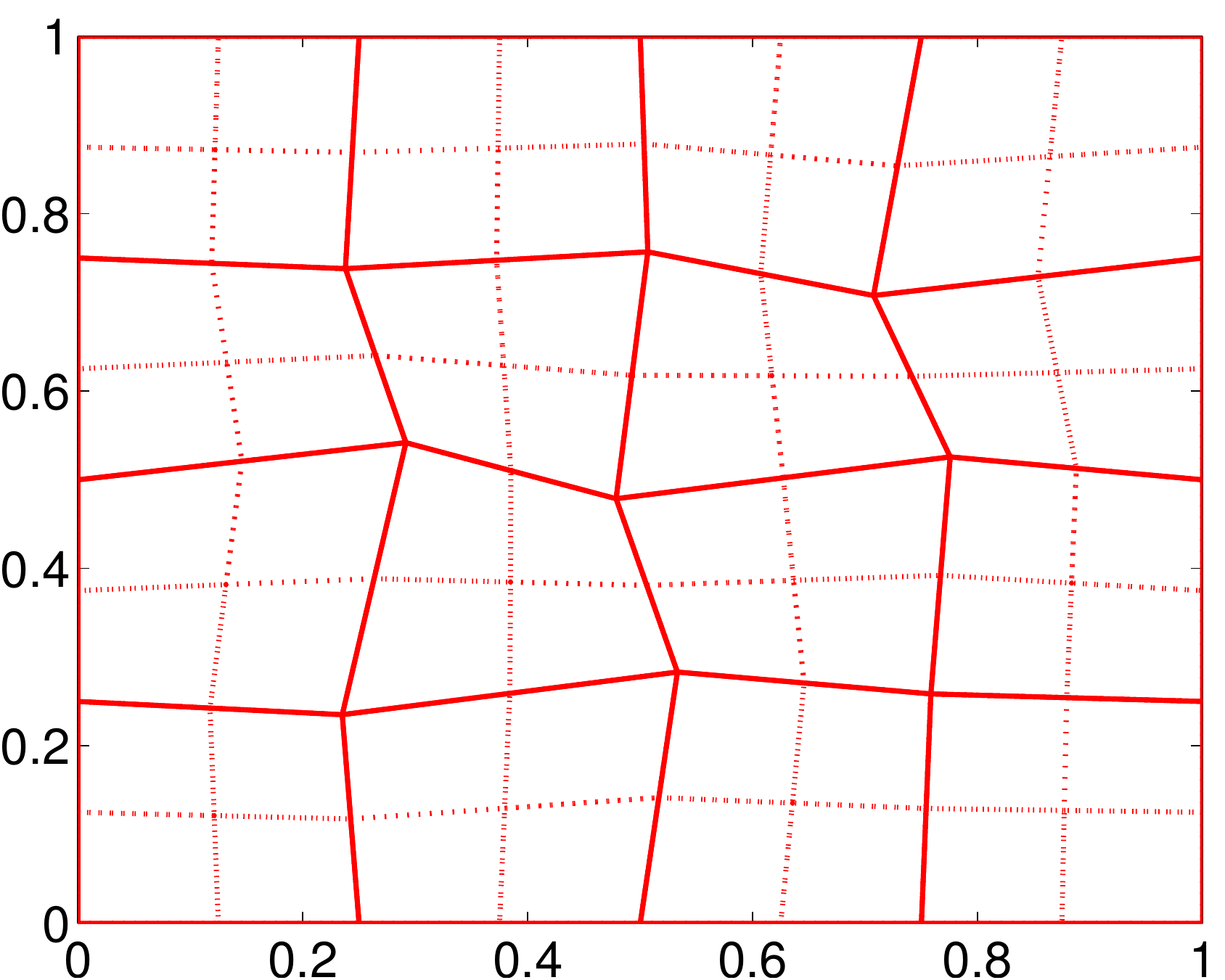}}
\end{tabular}
\caption{Mesh level 1 (Left) and mesh level 2 (Right) for example
3.}\label{fig:ex3}
\end{figure}

\begin{table}[h!]
\caption{Example 3. Convergence rate on quadrilateral
elements.}\label{table:ex3} \center
\begin{tabular}{||c||cc|cc|cc|cc||}
\hline\hline
$h$ & $\3bar\be_h\3bar$ & order & $\|\delta_h\|$& order  & $\|\epsilon_h\|_{1,h}$& order  & $\|\epsilon_h\|$& order \\
\hline\hline
   2.86e-1   &8.89e-1   &-         &{ 3.10e-2}   &-          &2.03     &     -    &8.08e-1   &-   \\ \hline
   1.43e-1   &4.74e-1   &{ 0.91}   &{ 6.30e-3}   &{ 2.30}       &1.95     &{ 0.56}   &2.71e-1   &1.58\\ \hline
   7.16e-2   &2.42e-1   &{ 0.97}   &{ 1.44e-3}   &{ 2.13}       &1.27     &{ 0.62}   &7.66e-2   &1.82\\ \hline
   3.58e-2   &1.21e-1   &{ 1.00}   &{ 3.51e-4}   &{ 2.04}   &7.08e-1  &{ 0.84}   &2.03e-2   &1.92\\ \hline
   1.79e-2   &6.03e-2   &1.01      &{ 8.71e-5}   &{ 2.01}      &3.71e-1  &{ 0.93}   &5.25e-3   &1.95\\ \hline
   8.95e-3   &2.95e-2   &1.03      &{ 2.17e-5}   &{ 2.00}       &1.90e-1  &{ 0.97}   &1.36e-3   &1.95\\ \hline\hline
\end{tabular}
\end{table}

\end{document}